\documentclass[a4paper,12pt]{amsart}

\usepackage[utf8]{inputenc}
\usepackage[english]{babel}
\usepackage{a4wide}

\usepackage{amsthm, amssymb}
\usepackage[alphabetic]{amsrefs}
\usepackage{enumitem}
\usepackage{color}

\usepackage[cmtip,arrow]{xy}
\usepackage{pb-diagram, pb-xy}

\usepackage{etoolbox}
\apptocmd{\sloppy}{\hbadness 10000\relax}{}{}

\newtheorem{theorem}{Theorem}[section]
\newtheorem{lemma}[theorem]{Lemma}
\newtheorem{proposition}[theorem]{Proposition}
\newtheorem{cor}[theorem]{Corollary}

\theoremstyle{definition}
\newtheorem{definition}[theorem]{Definition}
\newtheorem{example}[theorem]{Example}
\theoremstyle{remark}
\newtheorem{remark}[theorem]{Remark}

\newcommand{\CC}{\mathbb{C}}
\newcommand{\NN}{\mathbb{N}}
\newcommand{\ZZ}{\mathbb{Z}}
\newcommand{\RR}{\mathbb{R}}

\newcommand{\udisc}{\mathbb{D}}

\newcommand{\aut}{\mathop{\mathrm{Aut}}}
\newcommand{\id}{\mathop{\mathrm{id}}}
\newcommand{\Olo}{\mathcal{O}}
\newcommand{\hol}{\mathcal{O}}
\newcommand{\slgrp}{\mathrm{SL}}

\newcommand{\sball}{\Omega}
\newcommand{\sdisc}{\mathbb{G}}

\makeatletter
\@namedef{subjclassname@2020}{%
  \textup{2020} Mathematics Subject Classification}
\makeatother

\begin{document}

\author{Rafael~B.~Andrist}
\address{Rafael B. Andrist \\ Faculty of Arts and Sciences \\  American University of Beirut \\ Lebanon}
\email{ra332@aub.edu.lb}

\author{Riccardo~Ugolini}
\address{Riccardo~Ugolini \\ Faculty of Mathematics \\ Ruhr University Bochum \\ Germany}
\email{riccardo.ugolini@ruhr-uni-bochum.de}



\title{Tame sets in homogeneous spaces}

\begin{abstract}
We prove the existence of strongly tame sets in affine algebraic homogenenous spaces of linear algebraic Lie groups. We also show that $(\CC^n,A)$ for a discrete tame set enjoy the relative density property, and we provide examples of Stein manifolds admitting non-equivalent tame sets.
\end{abstract}

\maketitle

\section{Introduction}

The notion of a \emph{tame set} in $\CC^n$ was introduced by Rosay and Rudin in their seminal paper \cite{RosayRudin} in 1988.

\begin{definition}[\cite{RosayRudin}*{Def.~3.3}]
\label{def-tame-RR}
Let $e_1$ be the first standard basis vector of $\CC^n$. A set $A \subset \CC^n$ is called \emph{tame} (in the sense of Rosay and Rudin) if there exists an $F \in \aut(\CC^n)$ such that $F(A) = \NN \cdot e_1$. It is very tame if we can choose $F \in \aut_1(\CC^n)$ to be volume preserving.
\end{definition}

The significance of tame sets lies in their role as interpolating sets for holomorphic automorphisms. As an example, it is possible to find automorphisms with prescribed derivatives up to any order at every point of a tame set \cite{Forstneric1999}. It is even possible to do so with holomorphic dependence on a Stein parameter, under suitable topological conditions \cite{Ugolini2017}.

Since the definition of Rosay and Rudin makes explicit use of the fact that $\CC^n$ is a vector space, it is natural to ask how the notion of tameness can be generalized to complex manifolds. Independently, Winkelmann \cite{Winkelmann2019} and the authors \cite{AndristUgolini2019} introduced the notion of \emph{(weakly) tame} and \emph{(strongly) tame} in 2019.
 
\begin{definition}\cite{Winkelmann2019}
Let $X$ be a complex manifold.
An infinite discrete
subset $A$ is called \emph{(weakly) tame}
if for every exhaustion function $\rho \colon X \to \RR$
and every function $\zeta \colon A \to \RR$
there exists an automorphism $\Phi$ of $X$
such that $\rho(\Phi(x)) \geq \zeta(x)$ for all $x \in A$.
\end{definition}

\begin{definition}\cite{AndristUgolini2019}
\label{def-tame}
Let $X$ be a complex manifold and let $G \subseteq \aut(X)$ be a subgroup of its group of holomorphic automorphisms. We call a closed discrete infinite set $A \subset X$ a \emph{$G$-tame set} if for every injective mapping $f \colon A \to A$ there exists a holomorphic automorphism $F \in G$ such that $F|_{A} = f$.
\end{definition}

If $G$ is the full group of holomorphic automorphisms  $\aut(X)$ of a complex manifold $X$, we simply call a $G$-tame set a \emph{(strongly) tame} set.
It is easy to see that every strongly tame set is weakly tame, hence justifying the choices of the names. In general, the two notions are different; in fact, the manifold $\udisc \times \CC$ contains weakly tame sets, but not strongly tame ones \cite{Winkelmann2019}*{Proposition 5.3}. However, in $\CC^n$, both notions agree with the definition of Rosay and Rudin, see \cite{Winkelmann2019} and \cite{AndristUgolini2019}.

In the following, we will not considered weakly tame sets, hence \emph{tame} will always refer to \emph{strongly tame}.

It is simple to show that any complex manifold admitting a tame set is endowed with an infinite-dimensional group of holomorphic automorphisms \cite{AndristUgolini2019}*{Remark 7.3}. For this reason, we focus our search of tame sets to \emph{Stein manifolds with the density property}. This class includes all complex Lie groups and their homogeneous spaces of dimension at least $2$, excluding complex tori. They possess a very large group of holomorphic automorphisms; see \cite{FrancBook}*{Chapter 4} for more details and examples.

The authors generalized a classical interpolation result from \cite{RosayRudin} as follows.

\begin{proposition}[\cite{AndristUgolini2019}*{Proposition 2.4}] \label{dpequiv}
Let $X$ be a Stein manifold with the density property and let $A, B \subset X$ be tame sets. Then there exists a holomorphic automorphism $F \colon X \to X$ such that $F(A) = B$.
\end{proposition}

The authors showed the existence of tame sets in every connected complex-linear algebraic group different from the complex line $\CC$ or the punctured complex line $\CC^{\ast}$ \cite{AndristUgolini2019}.

Winkelmann in \cite{Winkelmann2020} extended the study of weakly tame sets to semi-simple Lie groups without non-trivial group homomorphisms to $\CC^\ast$ and gave several equivalent characterizations of weakly tame sets. In particular, he showed that in semi-simple Lie groups without a non-trivial group homomorphsims to $\CC^\ast$, the notions of tame and weakly tame coincide.

In this paper, we extend the existence result for tame sets to homogeneous spaces. Our main result is the following theorem.

\begin{theorem}
\label{mainthm}
Let $G$ be a complex-linear algebraic group and let $R \subset G$ be a closed reductive subgroup. Then the quotient $G/R$ contains tame subsets unless it consists only of connected components of $\CC$ or $\CC^\ast$.
\end{theorem}

Note that according to Matsushima's criterion \cite{Matsushima}, a quotient of a complex-linear algebraic group by a closed subgroup is affine if and only if the subgroup is reductive.

\bigskip

We will first construct tame sets in $\slgrp_2(\CC)$ and $\slgrp_2(\CC)/\CC^\ast$; this is the content of Section \ref{sl2case}.
In Section \ref{homtame}, we will use the structure theory of Lie groups to reduce to the case of a simple Lie group and then to $\slgrp_2(\CC)$, $\slgrp_2(\CC)/\CC^\ast$, or one of their finite quotients.
After proving our main Theorem \ref{mainthm}, we provide an example of a manifold admitting tame sets which cannot be mapped into each other via an automorphism, see Section \ref{nonequivtame}.
In Section \ref{secrelative} we prove that $(\CC^n, A)$ has the relative density property for a tame subset $A \subset \CC^n$. This provides another affirmative example to the problem mentioned in \cite{FrancBook}*{Problem 4.11.1} and \cite{FrancFrank}*{Section 2.3}.

\bigskip

We will repeatedly use complete holomorphic vector fields and their flows. Let us recall the following notions and facts.

\begin{remark}
Let $X$ be a complex manifold and let $V$ be a holomorphic vector field on $X$. We can interpret $V$ as a $\CC$-linear derivation $V \colon \hol(X) \to \hol(X)$. 
\begin{enumerate}
\item A vector field $V$ is called \emph{complete}, if its flow map $\varphi_t$ exists for all times $t \in \CC$. Since $\varphi_t$ satisfies the semi-group property, every complete holomorphic vector field generates a one-parameter family of holomorphic automorphisms $\varphi_t \colon X \to X$.
\item If $f \in \ker V$ and if $V$ is complete, then $f \cdot V$ is complete. We call $f \cdot V$ a \emph{shear} of $V$. If $V$ was volume-preserving\footnote{We call a vector field $V$ volume-preserving if $X$ is equipped with \emph{holomorphic volume form} $\omega$, i.e.\ $\omega$ is a nowhere degenerate, holomorphic form of maximal rank on $X$, and if $\mathrm{div}_\omega(V) = 0$.}, then so is $f \cdot V$.
\end{enumerate}
\end{remark}

\section{The group $\slgrp_2(\CC)$ and its homogeneous spaces} \label{sl2case}

According to Mostow's decomposition, an algebraic group can be decomposed into semi-direct product a unipotent group and a reductive group which further decomposes into a semi-simple and toric part. It is therefore essential to understand the semi-simple part and its building blocks, namely $\slgrp_2(\CC)$. In this section we study $\slgrp_2(\CC)$ and its homogeneous spaces.

We quote from Donzelli--Dvorsky--Kaliman \cite{DDK2010}*{Lemma 13}:
\begin{lemma}
Let $X$ be a smooth complex affine algebraic variety equipped with
a fixed point free non-degenerate $\slgrp_2(\CC)$-action and let $x \in X$ be a point contained in a closed $\slgrp_2(\CC)$-orbit. Then the isotropy group of $x$ is either finite, or isomorphic to the diagonal $\CC^\ast$-subgroup of $\slgrp_2(\CC)$, or to the normalizer of this $\CC^\ast$-subgroup (which is the extension of $\CC^\ast$ by $\mathbb{Z}_2$).
\end{lemma}

Note that the quotient by the normalizer mentioned above is isomorphic to the quotient of the Danielewski hypersurface $\{z^2 -z -xy = 0\}$ obtained by identifying $(x,y,z)$ with $(-x,-y,-z+1)$; see also \cite{DDK2010}*{Section 3.1}.
Since the Danielewski hypersurface is isomorphic to the quotient $\slgrp_2(\CC) / \CC^\ast$, we only need to consider finite quotients of $\slgrp_2(\CC)$ and of $\slgrp_2(\CC) / \CC^\ast$.

Before proceeding, let us recall the following well-known result; an elementary proof can be found in the textbook of Harris \cite{Harris}*{page 124}.
\begin{lemma} \label{affquotient}
Let $X$ be an affine algebraic variety with a finite group $E$ acting by algebraic automorphisms. Then the quotient $X / E$ is an affine algebraic variety as well.
\end{lemma}

\subsection{The group $\slgrp_2(\CC)$ and its finite quotients}

The existence of tame sets in $\slgrp_2(\CC)$ was established in \cite{AndristUgolini2019} where the proof makes use of complete holomorphic vector fields arising from left and right multiplication alike. In order to prove existence of tame sets in homogeneous spaces, this is not sufficient; we need to find tame sets in $\slgrp_2(\CC)$ using only automorphisms originating from left multiplication.

We will use the following three complete vector fields that are the standard generators of the Lie algebra $\mathfrak{sl}_{2}(\mathbb{C})$ and that correspond to left-multiplications:
\begin{align*}
V
&= c \frac{\partial }{\partial a} + d \frac{\partial }{\partial b}
\\
W
&= a \frac{\partial }{\partial c} + b \frac{\partial }{\partial d}
\\
U
&= -a \frac{\partial }{\partial a} -b \frac{\partial }{\partial b}
+c \frac{\partial }{\partial c} +d \frac{\partial }{\partial d}
\end{align*}
satisfying $[V, W] = U$, $[U,V] = 2V$ and $[U, W] = 2W$, with their
respective flows
\begin{align*}
\varphi _{V}^{t}
&=
\begin{pmatrix}
1 & t
\\
0 & 1
\end{pmatrix}
\cdot
\begin{pmatrix}
a & b
\\
c & d
\end{pmatrix}
&=
\begin{pmatrix}
c t+a & d t+b\\
c & d
\end{pmatrix}
\\
\varphi _{W}^{t} &=
\begin{pmatrix}
1 & 0
\\
t & 1
\end{pmatrix}
\cdot
\begin{pmatrix}
a & b
\\
c & d
\end{pmatrix}
&=
\begin{pmatrix}
a & b\\
a t+c & b t+d
\end{pmatrix}
\\
\varphi _{U}^{t} &=
\begin{pmatrix}
e^{-t} & 0
\\
0 & e^{t}
\end{pmatrix}
\cdot
\begin{pmatrix}
a & b
\\
c & d
\end{pmatrix}
&=
\begin{pmatrix}
a\, {e}^{-t} & b\, {e}^{-t}\\
c\, {e}^{t} & d\, {e}^{t}
\end{pmatrix}
\end{align*}

We also have the following relations for the kernels:
\begin{align*}
\ker {V}
&\supseteq \mathbb{C}[c, d]
\\
\ker {W}
&\supseteq \mathbb{C}[a, b]
\\
\ker {U}
&\supseteq \mathbb{C}\left<ac, ad, bc, bd\right>
\end{align*}

Note that $ad - bc = 1$ is contained in all the kernels.

As we already mentioned, we will work with shears of these vector fields. If $f \colon \slgrp_2(\CC) \to \CC$ is a holomorphic function in the kernel of $V$, then $f \cdot V$ is also a complete vector field and its time-$1$ flow is given by
\[
\varphi^f_V = \begin{pmatrix}
1 & f
\\
0 & 1
\end{pmatrix}
\cdot
\begin{pmatrix}
a & b
\\
c & d
\end{pmatrix}.
\]

Similar formulas hold for $W$ and $U$ as well.

\begin{proposition}
\label{prop:sl2tame}
The set $\left\{ \begin{pmatrix}
1 & n \\
0 & 1 
\end{pmatrix}\right\}_{n \in \NN} \subset \slgrp_2(\CC)$ is $G$-tame, where $G \subset \aut(\slgrp_2(\CC))$ is the group generated by the flows of $V, W, U$ and their shears.
\end{proposition}

\begin{proof}
Given any injection $\ell \colon \NN \to \NN$ we need to show that we can find a holomorphic automorphism $F \colon \slgrp_2(\CC) \to \slgrp_2(\CC)$ with $F\left( \begin{pmatrix}
1 & n \\
0 & 1 
\end{pmatrix} \right) = \begin{pmatrix}
1 & \ell(n) \\
0 & 1 
\end{pmatrix}$.

We make the following ansatz for $F$:
\[
\begin{pmatrix}
1 & n \\
0 & 1 
\end{pmatrix}
\mapsto
\begin{pmatrix}
1 & 0 \\
g & 1
\end{pmatrix}
\cdot
\begin{pmatrix}
\mu^{-1} & 0 \\
0 & \mu
\end{pmatrix}
\cdot
\begin{pmatrix}
1 & f \\
0 & 1
\end{pmatrix}
\cdot
\begin{pmatrix}
1 & 0 \\
n+1 & 1 
\end{pmatrix}
\cdot
\begin{pmatrix}
1 & n \\
0 & 1 
\end{pmatrix}
=
\begin{pmatrix}
1 & \ell(n) \\ 0 & 1
\end{pmatrix}.
\]

The functions $f, g$ and $\mu$ will be in the kernels of the respective vector fields.

For the first two multiplications we obtain
\[
\begin{split}
P(n) := 
\begin{pmatrix} 1 & f(n) \\ 0 & 1 \end{pmatrix} \cdot
\begin{pmatrix} 1 & n \\ n+1 & n \cdot \left( 1+n\right) +1 \end{pmatrix}
\\ =
\begin{pmatrix} f(n) \cdot \left( 1+n\right) +1 & f(n) \cdot \left( 1+n\cdot \left( n+1\right) \right) +n\\\
 n+1 & n\cdot \left( 1+n\right) +1 \end{pmatrix}
\end{split}
\]
We can find an interpolating holomorphic $f = f(c)$ with the following values for $n \in \NN$:
\[
f(n+1) = \frac{n-\ell(n)}{(1+n) \cdot \ell(n)-n^2-n-1}
\]
and obtain
\[
P(n) = 
\begin{pmatrix}
\frac{1}{{{n}^{2}}-\ell(n) n+n-\ell(n)+1} & \frac{\ell(n)}{{{n}^{2}}-\ell(n) n+n-\ell(n)+1}\\[2mm]
n+1 & {{n}^{2}}+n+1
\end{pmatrix}
\]
This would only pose a problem if $\ell(n) = \frac{1+n+{{n}^{2}}}{n+1} \in \NN$, which occurs only if $(n+1)$ would divide $n^2$, but this never happens.

In order to find $\mu$, we need first to recover $n$ from the functions $bc$ and $bd$. For all $n \in \NN$ and prescribed $\ell(n) \in \NN$ we observe that the map to $\CC^2$ given by $(a,b,c,d) \mapsto (bc,bd)$ is injective on $P(\NN)$, since $bd/bc = n + \frac{1}{n+1}$ is injective for $n \in \NN$.
Moreover, $(a,b,c,d) \mapsto (bc,bd)$
has no accumulation points for $(a,b,c,d) \in P(\NN)$:
note first that since $\ell \colon \NN \to \NN$ is an injective map, it is unbounded on any subsequence. Then, $bd$ tends to $\infty$ as $n \to \infty$. To see this, consider
\[
\frac{1}{bd} = \frac{n^2 + n + 1 - \ell(n) \cdot (n + 1)}{\ell(n) \cdot (n^2 + n + 1)} = \frac{1}{\ell(n)} - \frac{n + 1}{n^2 + n + 1} \to 0
\]
Hence we can recover $n$ from $P(n)$ as a holomorphic function of $bc$ and $bd$.

Then, we choose $\mu(n) = f(n) \cdot \left( 1+n\right) +1$ and obtain:
\[
\begin{pmatrix}
\mu^{-1} & 0 \\
0 & \mu
\end{pmatrix}
\cdot P(n) = \begin{pmatrix}1 & \frac{n+f(n) \cdot \left( n\cdot \left( 1+n\right) +1\right) }{f(n) \cdot \left( 1+n\right) +1} \\[2mm]
 \left( n+1\right) \cdot \left( f(n) \cdot \left( 1+n\right) +1\right)  & \left( f(n) \cdot \left( 1+n\right) +1\right) \cdot \left( n\cdot \left( 1+n\right) +1\right) \end{pmatrix}
\]
 
By injectivity of $\ell$ we can find now a holomorphic $g$ with 
\[
g(n) = -\left( n+1\right) \cdot \left( f(n) \cdot \left( 1+n\right) +1\right).
\]
It is now trivial to compute the last matrix multiplication and verify the claim.
\end{proof}

\begin{remark} \label{sl2finite}
We now explain how to modify the proof in case of a finite quotient of $X=\slgrp_2(\CC)$.

Suppose $A \subset X$ is a tame set that was constructed only using a finite composition of complete flows $\varphi^t_U, \varphi^t_V, \varphi^t_W$ as in Proposition \ref{prop:sl2tame}. We want to show that $\pi(A)$ is a tame set in $X / E$, where $E$ is a finite group and $\pi \colon X \to X / E$ is the quotient projection.  

It is clear that $\pi(A)$ is an countable set, and that it is closed and discrete.
Let $\ell \colon \pi(A) \to \pi(A)$ be any permutation.
Let $A = \left\{ \begin{pmatrix} 1 & n \\ 0 & 1 \end{pmatrix} \,:\, n \in \NN \right\}$. For each element of $\pi(A)$ we pick one of the finitely many representatives. We obtain a sequence $(n_m)_{m \in \NN}$ such that $A^\prime = \left\{ \begin{pmatrix} 1 & n_m \\ 0 & 1 \end{pmatrix} \,:\, m \in \NN \right\}$ is mapped bijectively to $\pi(A)$ by $\pi$.
We follow now the proof of Proposition \ref{prop:sl2tame}, except that we replace $\ell \colon \NN \to \NN$ by $\ell \colon \{n_m \,:\, m \in \NN\} \to \{n_m \,:\, m \in \NN\}$ and use the corresponding vector fields and flows on the quotient $X / E$. Note that these are well-defined as they come from left multiplication.

The functions that have to be interpolated are chosen on the quotient, which is affine by Lemma \ref{affquotient}. In practice, we will need $f,g$, and $\mu$ to be $E$-invariant. Let us explain how to obtain such invariant $f \in \ker V$; $g$ and $\mu$ are obtained similarly.

Given a function $\tilde{f} \in \ker V$, we can define the average
\[
f(x) := \frac{1}{\# E} \sum_{e \in E} \tilde{f}(x \cdot e)
\]
as it is done e.g.\ in Harris \cite{Harris}*{page 124}.

The function $f$ is still in $\ker V$ and it is clearly $E$-invariant. To interpolate $f$ on the set $A'$ it is then sufficient to interpolate $\tilde{f}$ on the set $\{ a_{n_m} \cdot e \,:\, m \in \NN, \, e \in E \}$.  If at least two points in the $E$-orbit $A_m = \{ a_{n_m} \cdot e \,:\, e \in E \}$ can be distinguished by functions in $\ker V$ for every $m \in \NN$, it is clear that we can prescribe values to $f$. On the other hand if for some $m \in \NN$ every function in $\ker V$ is constant on the set $A_m$, then $f(a_m \cdot e) = \tilde{f}(a_m)$ and it is sufficient to interpolate at just one point.
\end{remark}

\subsection{The homogeneous space $\slgrp_2(\CC) / \CC^\ast$ and its finite quotients}

Here, we consider the homogeneous space $\slgrp_2(\CC) / \CC^\ast$ and prove that it admits tame sets.

The standard embedding of $\CC^\ast$ in $\slgrp_2(\CC)$ is given by
$\CC^\ast \ni \lambda \mapsto \begin{pmatrix} \lambda & 0 \\ 0 & \lambda^{-1} \end{pmatrix} \in \slgrp_2(\CC)$.

Given the group multiplication $ \begin{pmatrix} a & b \\ c & d \end{pmatrix}  \cdot \begin{pmatrix} \lambda & 0 \\ 0 & \lambda^{-1} \end{pmatrix}  =  \begin{pmatrix} \lambda a & \lambda^{-1} b \\ \lambda c & \lambda^{-1} d \end{pmatrix} $
it is easy to deduce the invariant functions
$ z := a d, 
x := a b, 
y := c d, 
w := b c = z - 1 $
and the relation
\begin{equation}
\label{definingeq}
z^2 - z - x y = 0
\end{equation}

The left multiplications on $\slgrp_2(\CC)$ descend to the quotient as follows:
\begin{align*}
\begin{pmatrix} 1 & t \\ 0 & 1 \end{pmatrix} \cdot \begin{pmatrix} a & b \\ c & d \end{pmatrix} 
&\rightsquigarrow \begin{pmatrix} z \\ x \\ y \end{pmatrix} \mapsto \begin{pmatrix} z + t y \\ x + t(2 z - 1) + t^2 y \\ y \end{pmatrix}
&\rightsquigarrow y \frac{\partial}{\partial z} + (2 z - 1) \frac{\partial}{\partial x} \\
\begin{pmatrix} 1 & 0 \\ t & 1 \end{pmatrix} \cdot \begin{pmatrix} a & b \\ c & d \end{pmatrix} 
&\rightsquigarrow \begin{pmatrix} z \\ x \\ y \end{pmatrix} \mapsto \begin{pmatrix} z + t x \\ x \\ y + t (2 z - 1) + t^2 x \end{pmatrix}
&\rightsquigarrow x \frac{\partial}{\partial z} + (2 z - 1) \frac{\partial}{\partial y} \\
\begin{pmatrix} e^t & 0 \\ 0 & e^{-t} \end{pmatrix} \cdot \begin{pmatrix} a & b \\ c & d \end{pmatrix} 
&\rightsquigarrow \begin{pmatrix} z \\ x \\ y \end{pmatrix} \mapsto \begin{pmatrix} z \\ e^{2 t} x \\ e^{-2t} y \end{pmatrix}
&\rightsquigarrow 2 \frac{\partial}{\partial x} - 2 \frac{\partial}{\partial y}
\end{align*}
These are the well-known ``basic'' complete algebraic vector fields on the Danielewski surface $\{(x,y,z) \in \CC^3 \,:\, z^2 - z - x y = 0 \}$. They are in particular well-defined vector fields which annihilate the outer differential of \eqref{definingeq}.

It is straightforward that $y^k \in \ker y \frac{\partial}{\partial z} + (2 z - 1) \frac{\partial}{\partial x}$ and that $x^m \in \ker x \frac{\partial}{\partial z} + (2 z - 1) \frac{\partial}{\partial y}$ for all $k, m \in \NN_0$. Moreover, $z y^k$ resp.\ $z x^m$ are in the second kernels. 

\begin{proposition}
\label{prop:tamdani}
A tame set in the Danielewski surface $\{(x,y,z) \in \CC^3 \,:\, z^2 - z - x y = 0 \}$ is given by $A = \{ (n, 0, 1) \, : \, n \in \NN \}$ 
\end{proposition}

Note that elements in $\slgrp_2(\CC)$ representing these equivalence classes are e.g.\ $\begin{pmatrix} 1 & n \\ 0 & 1 \end{pmatrix}, \; n \in \NN,$ which form a tame set by Proposition \ref{prop:sl2tame}.

\begin{proof} 
The flow map of $y \frac{\partial}{\partial z} + (2 z - 1) \frac{\partial}{\partial x}$ is given by 
\[
\varphi_t(x,y,z) = (x + t (2 z - 1) + t^2 y, y, z + t y)
\]

The flow map of $x \frac{\partial}{\partial z} + (2 z - 1) \frac{\partial}{\partial y}$ is given by 
\[
\psi_t(x,y,z) = (x, y + t (2 z - 1) + t^2 x, z + t x)
\]

Our desired automorphism that maps $(n, 0, 1)$ to $(\ell(n), 0, 1)$, where $\ell \colon \NN \to \NN$ is an injective map, is given as a composition of three automorphisms:
\begin{enumerate}[label=(\Roman*)]
\item $(x, y, z) \mapsto \psi_1(x,y,z)$, in particular
\[
(n, 0, 1) \mapsto (n, n+1, n+1)
\]
\item $(x, y, z) \mapsto \varphi_{f(y-1)}(x,y,z)$, in particular
\[
(n, n+1, n+1) \mapsto (\ell(n), n+1, n+1 + f(n) \cdot (n+1))
\]
where we have chosen $f \colon \CC \to \CC$ to be a holomorphic function such that
\[
\ell(n) = n + f(n) \cdot (2n+1) + f^2(n) \cdot (n+1) \; \text{ for all } n \in \NN
\]
\item $(x, y, z) \mapsto \psi_{g(x)}(x, y, z)$, in particular
\[
(\ell(n), n+1, n+1 + f(n) \cdot (n+1)) \mapsto (\ell(n), ?, (n+1) \cdot (f(n) + 1) + g(\ell(n)) \cdot \ell(n))
\]
\end{enumerate}

Since $\ell$ is injective and since the image of $\NN$ under $\ell$ is closed and discrete as well, a holomorphic $g \colon \CC \to \CC$ can be constructed such that the $z$-coordinate assumes the value $1$ for all $n \in \NN$. Note that the $y$-coordinate necessarily will be zero, since all the maps were well-defined maps on the Danielewski surface and the coordinates need to satisfy $z^2 - z = x y$.
\end{proof}

\begin{remark}
\label{danifinite}
We now explain how to modify the proof in case we consider the quotient of $\slgrp_2(\CC)/\CC^\ast$ by $\ZZ/2\ZZ$.
This corresponds to the quotient of the hypersurface by Equation \eqref{definingeq} where we identify $(x,y,z)$ with $(-x,-y,-z+1)$, see also \cite{DDK2010}*{Section 3.1}.
The tame set constructed in the proposition above was $A = \{ (n, 0, 1) \, : \, n \in \NN \}$. Note that $(n, 0, 1)$ and $(-n, 0, 0)$ are in the same equivalence class. Inspection of the steps in the proof show that the $x$-coordinate is always a natural number, hence none of the images of $A$ intersects with their counterparts in the respective equivalence class. Therefore, we can use the same vector fields and the representatives given by $A$ to obtain the tame set in this quotient, and prescribing the values of the functions $f$ and $g$ will be well-defined on the quotient.
\end{remark}

\section{Existence of tame sets in affine homogeneous spaces} \label{homtame}

Let $X$ be an affine algebraic complex manifold equipped with a $G=\slgrp_2(\CC)$ algebraic action. Denote by $H_1$ and $H_2$ the unipotent subgroups of $G$ of upper and lower triangular matrices respectively with ones on the diagonal. The group $H_1$ corresponds to the vector field $V$, while $H_2$ corresponds to $W$. We also denote by $H_3$ the closed subgroup of diagonal matrices corresponding to the vector field $U$ in $G$.

\begin{theorem}
\label{findshears}
Assume there exists a point $x_0 \in X$ such that its orbit $G\cdot x_0 \subset X$ is closed and let $H$ be one of the above subgroups. Every $H$-invariant holomorphic function on  $G\cdot x_0 \subset X$ extends to a $H$-invariant holomorphic function on $X$.
\end{theorem}

\begin{proof}
Consider the $G \times H$-action on $G \times X$ given by
\[
\begin{split}
(G \times H) \times (G \times X) & \to G \times X \\
(g,h),(a,x)& \mapsto (hag^{-1},g \cdot x)
\end{split}
\]

The \textit{transfer principle}  \cite{Grosshans}*{Chap.~Two} states that the ring of $H$-invariant functions $\CC[X]^H$ is isomorphic to $\CC[H \backslash G \times X]^G$, where $H \backslash G$ denotes the space of right cosets for the $H$-action given by left multiplication on $G$. In particular, we are interested in the fact that given a holomorphic function $f \in \Olo(G / G_{x_0})^H$, we can define a holomorphic $H$-invariant function $F$ on $G\cdot x_0 \ \subset X$ as 
\[
F(g\cdot x_0)=f(gG_{x_0}).
\]

Since $F$ is $H$-invariant, it descends to a holomorphic function on $\pi_H(G\cdot x_0) \subset X // H$, where the latter is the categorical quotient of $X$ given by the restriction of the $G$-action to $H$. The set $\pi_H(G\cdot x_0)$ is closed in $X // H$ because $G\cdot x_0$ is closed and $H$-saturated, and $X \to X // H$ is a quotient map. Hence, we can extend $F$ to a $H$-invariant holomorphic function on $X$.
\end{proof}

For a complex $G$-manifold $X$, denote by $G(X) \subset \aut(X)$ the group generated by the $H_1, H_2$ and $H_3$ actions and their shears. Theorem \ref{findshears} guarantees the following: 

\begin{cor} \label{larger}
Let $X$ be a complex manifold equipped with an algebraic $G$-action admitting a closed orbit $C$. If $A \subset C$ is $G(C)$-tame in $C$, then it is $G(X)$-tame in $X$. 
\end{cor}

\begin{proof}
Let $A = \{ a_n \}_{n \in \NN} \subset C$ and let $\ell \colon \NN \to \NN$ be a prescribed injection. Since $A$ is $G(C)$-tame, there exists automorphisms $F_k \in G(C)$ for $k = 1, \dots, N$, for some $N \in \NN$ such that $F_N \circ \dots \circ F_1 (a_n) = a_{\ell(n)}$ for all $n \in \NN$ and $F_k = \varphi^f_Y$ for $Y \in \{V, W, U\}$ and $f \in \ker Y \subset \Olo(C)$.

By Theorem \ref{findshears}, $f$ can be extended to a holomorphic function (which we still denote by $f$) on $X$ which is still in the kernel of $Y$. Then each $F_k = \varphi^f_Y$ is a well-defined automorphism of $X$ which extends the one on $C$. Hence, $A$ is $G(X)$-tame.
\end{proof}

\bigskip

We will need the following theorem from our previous work.
\begin{theorem}[\cite{AndristUgolini2019}*{Theorem~5.10}]
\label{commuting}
Let $X$ be an affine algebraic complex manifold and let $V,W$ be
complete algebraic vector fields whose flows are algebraic. If
$[V,W]=0$ and their kernels are not contained one into the other, then
there exists a tame set in $X$.
\end{theorem}

We will now prove Theorem \ref{mainthm} using Theorems \ref{findshears} and \ref{commuting} above.

\begin{proof}[Proof of Theorem \ref{mainthm}]
We closely follow Donzelli--Dvorsky--Kaliman \cite{DDK2010}*{Corollary~25 and Theorem~26} to find either a non-degenerate fix-point free action of $\slgrp_2(\CC)$ on $X = G/H$ that arises from the left-multiplication of $G$ or a pair of vector fields as in Theorem \ref{commuting}.

We can exclude the case that $G$ is isomorphic to $\slgrp_2(\CC)$ since this case and all its homogeneous spaces have been treated already.

Let $G$ be a linear algebraic group. Consider the nontrivial unipotent radical $U$ of $G$. By Mostov's decomposition theorem there is a maximal closed reductive subgroup $G_0$ of $G$ such that $G$ is a semidirect product of $U$ and $G_0$. We can suppose that $R$ is contained in $G_0$, see proof of \cite{DDK2010}*{Theorem~26, p.~568}, and conclude that $G/R$ is isomorphic as algebraic variety to $G_0/R \times U$. If both factors are non-trivial, we pick natural non-trivial $\CC^+$-actions in each factor and we are done by Theorem \ref{commuting}.
If we have only the factor $U$, then it is isomorphic as algebraic variety to $\CC^n$ and we are done for $n \geq 2$ by the classical results on existence of tame sets.

Hence, we can assume in the following that $G$ is isomorphic to $G_0/R$ with reductive $G_0$ which is isomorphic to $(S \times (\CC^\ast)^m)/E$ where $S$ is a complex semisimple group and $E$ is a finite central subgroup. Let us first discuss the case $G = S \times (\CC^\ast)^m$ i.e. $E$ is trivial. The proof of \cite{DDK2010}*{Theorem~26, p.567} gives that our homogeneous space is isomorphic to a product $S/R' \times (\CC^\ast)^{m_0}$ for a reductive group $R'$ and $m_0 \leq m$. If neither factor is trivial, we again apply Theorem \ref{commuting}. If we only have a torus, then the existence of tame sets was proved in \cite{AndristUgolini2019}. The case of a homogeneous space of a semi-simple group will be covered later in the proof.

If $E$ is not trivial, the fact that it is central guarantees that any action on $G$ coming from left multiplication descends to the quotient. The projection of a tame set in $G$ will then be tame, since all our automorphisms come from left multiplication; finding $E$-invariant interpolating functions is handled as in Remark \ref{sl2finite}.

Let us now deal with the case of $G$ semi-simple; following the proof of \cite{DDK2010}*{Corollary~25}, we have that $X$ is isomorphic to a quotient of the form $G/R$ where $G = G_1 \oplus \dots \oplus G_N$ is semi-simple and the $G_i$'s are simple.
Consider the projection $\pi_k \colon G \to G_k$ and $R_k = \pi_k(R)$ which is reductive being the image of a reductive group. If $N$ is the minimal possible then $R_k \neq G_k$ for every $k$.
If $N \geq 2$, we again conclude with Theorem \ref{commuting}.
Hence, we assume that $G$ is a simple Lie group. If it is isomorphic to $\slgrp_2(\CC)$, we are done by Section \ref{sl2case}. Otherwise, \cite{DDK2010}*{Theorem~24} provides the existence of a subgroup of $G$ which is isomorphic to $\slgrp_2(\CC)$ and such that its action by left multiplication on $G/R$ is non-degenerate and fixed-point free. In particular, this action admits a closed orbit and we conclude by applying Corollary \ref{larger} and the fact that such an orbit admits a tame set, which was proven in Section \ref{sl2case}.
\end{proof}

\section{Non-equivalent tame sets} \label{nonequivtame}
Theorem \ref{dpequiv} gives that in a Stein manifold with the density property, any two tame sets are equivalent. We are interested in finding a Stein manifold admitting non-equivalent tame sets. The presence of a tame set forces the automorphism group of a manifold to be somewhat large, while the existence of tame sets that cannot mapped into each other restricts the size of the automorphisms group; we here provide a few examples that lie between these two conditions.

Before proceeding, we observe that the simple case of $\udisc \times \CC$ does not fulfill our purpose as it simply does not admit tame sets (even if it does admit weakly tame ones). Fortunately, we do not have to look much further.

\begin{proposition} \label{nontameproduct}
Let $\Omega$ be a Brody hyperbolic domain (i.e.\ all holomorphic maps from $\CC$ into $\Omega$ are constant) such that $\aut(\Omega)$ does not act transitively. Then, $\Omega \times \CC^n$ admits non-equivalent tame sets for all $n>1$.
\end{proposition}

Examples of such a domains $\Omega$ are all non-homogeneous bounded domains $\Omega \subset \CC^n$.

Proposition \ref{nontameproduct} is a direct consequence of the following Lemma.

\begin{lemma}
Let $\Omega$ be a Brody hyperbolic domain. Then, all automorphisms of $\Omega \times \CC^n, \ n>1$ are of the form $(z,w) \mapsto (f(z),g(z,w))$ where $f \in \aut(\Omega)$.
\end{lemma}

\begin{proof}
Let $F \in \aut(\Omega \times \CC^n)$ and let $\pi_1 \colon \Omega \times \CC^n \to \Omega$ be the projection onto the first coordinate.

Fixing $z \in \Omega$, hyperbolicity gives that the map $w \mapsto \pi_1 \circ F (z,w)$ is constant i.e.\ $F(z,w)=(f(z), g(z,w))$.
As $F$ is an automorphism, $f$ is surjective. We are left to prove that $f$ is also injective: 
Assume it is not and let $z \in \Omega$ be such that $A = \{ z' \in \Omega : f(z) = f(z') \} \neq \{z\}$. For each fixed $z' \in \Omega$, the map $w \mapsto g(z',w)$ is injective. Furthermore for $z' \neq z'' \in A$ we have that $g(z', \CC^n) \cap g(z'',\CC^n) = \emptyset$; if that was not the case, $F$ would not be injective.

Hence we have the decomposition $\CC^n = \bigsqcup_{z' \in A} g(z',\CC^n)$ into Fatou--Bieberbach domains. This is not possible unless $A = \{z\}$, reaching a contradiction.
\end{proof}

A more sophisticated example of a Stein manifold admitting non-equivalent tame sets which is not a direct product, is the spectral ball:

\begin{definition}
The $n$-dimensional \emph{spectral ball} $\sball_n$ is defined as
\[
\sball_n := \left\{ A \in \mathrm{Mat}(n \times n, \CC) \,:\, \mathop{\mathrm{spec}} A \subset \udisc \right\}
\]
\end{definition}
In dimension $n = 1$ the spectral ball is just the unit disc. For higher dimensions, the automorphism group of $\sball_n$ has been described by Andrist and Kutzschebauch \cite{AndristKutzschebauch2018}.
There is a natural map $\pi \colon \mathrm{Mat}(n \times n, \CC) \to \CC^n$ which sends a square matrix to its symmetrized vector of eigenvalues. In dimension $n = 2$, we observe that $\pi = (\mathrm{tr}, \det)$. The image $\pi(\sball_2) = \sdisc^2$ is the symmetrized polydisc, a bounded and hence hyperbolic domain in $\CC^2$. The fibers of $\pi$ consist precisely of the matrices with the same eigenvalues. A fibre with two different eigenvalues is homogeneous w.r.t.\ conjugation by $\slgrp_2(\CC)$. A fibre with a double eigenvalue consists of two orbits w.r.t.\ conjugation by $\slgrp_2(\CC)$ and is singular. In any case, every fibre contains copies of $\CC$ (and is in fact $\CC$-connected), as can be seen from $\CC \ni t \mapsto \begin{pmatrix} \lambda & t \\ 0 & \mu \end{pmatrix} \in \sball_2$ with eigenvalues $\lambda, \mu \in \udisc$. Since $\sdisc^2$ is hyperbolic, every composition of such a map with $\pi$ is constant. Hence an automorphism of $\sball_2$ necessarily maps fibres to fibres. However, singular fibres cannot be mapped to non-singular fibres by a holomorphic automorphism of $\sball_2$. If we can construct tame sets in each fibre, then we have constructed non-equivalent tame sets in $\sball_2$. Note that we just proved as well, that the action of holomorphic automorphisms of $\sball_2$ is not transitive. The question, whether a Stein manifold admitting tame sets and having a transitive action of the group of holomorphic automorphism, has the density property, remains open.

\begin{lemma}
Each fibre of $\pi \colon \sball_2 \to \sdisc^2$ contains an $\aut(\sball_2)$-tame set.
\end{lemma}

\begin{proof}
We show that for each $\lambda, \mu \in \udisc$, including the case $\lambda = \mu$, the set 
\[
\left\{ \begin{pmatrix}
\lambda & k \\
0 & \mu
\end{pmatrix}
\;:\;
k \in \NN
\right\} 
\subset \pi^{-1}(\lambda + \mu, \lambda \cdot \mu) \subset \sball_2
\]
is tame.

In the following we will fix the coordinates for
\[
\slgrp_2(\CC) = \left\{ \begin{pmatrix} a & b \\ c & d \end{pmatrix} \,:\, ad - bc = 1 \right\}
.\]

We will need the following $\CC$-complete vector fields
\begin{align*}
V &= c \frac{\partial}{\partial a} + (d-a) \frac{\partial}{\partial b} - c  \frac{\partial}{\partial d} \\
W &= -b \frac{\partial}{\partial a} + (a-d) \frac{\partial}{\partial c} + b \frac{\partial}{\partial d}
\end{align*}
with their respective flows
\begin{align*}
\varphi_V^t &= \begin{pmatrix} 1 & t \\ 0 & 1 \end{pmatrix} \cdot \begin{pmatrix} a & b \\ c & d \end{pmatrix} \cdot\begin{pmatrix} 1 & -t \\ 0 & 1 \end{pmatrix} = \begin{pmatrix}c t+a & t \left( d-c t\right) -a t+b\\ c & d-c t\end{pmatrix} \\
\varphi_W^t &= \begin{pmatrix} 1 & 0 \\ t & 1 \end{pmatrix} \cdot \begin{pmatrix} a & b \\ c & d \end{pmatrix} \cdot \begin{pmatrix} 1 & 0 \\ -t & 1 \end{pmatrix} = \begin{pmatrix}a-b t & b \\ t \left( a-b t\right) -d t+c & b t+d\end{pmatrix}\end{align*}

We also have the following relations for the kernels:
\begin{align*}
\ker {V} &\supseteq \CC[c,a+d] \\
\ker {W} &\supseteq \CC[b,a+d] 
\end{align*}

We are given an injective map $\ell \colon \NN \to \NN$ that prescribes the injection
\[
\begin{pmatrix}
\lambda & k \\
0 & \mu
\end{pmatrix}
\mapsto
\begin{pmatrix}
\lambda & \ell(k) \\
0 & \mu
\end{pmatrix}
.\]

Next, we construct an interpolating holomorphic automorphism of $\slgrp_2(\CC)$ as a composition of suitable time-$1$ maps of complete vector fields corresponding to conjugations.

The desired automorphism is given by:
\[
\varphi_W^{G} \circ \varphi_V^{F}  \circ \varphi_W^1 
\]
where $F$ and $G$ are holomorphic functions given as follows:
\begin{enumerate}
\item \[ F\left( \begin{pmatrix} a & b \\ c & d \end{pmatrix} \right) = f(\lambda-\mu-c) \]
By the Mittag-Leffler interpolation theorem we find a holomorphic function $f \colon \CC \to \CC$ such that $f(k)$ solves the quadratic equation
\[
-(1+f(k)) \cdot (-k - k f(k) + (\lambda - \mu) f(k)) = \ell(k)
\]
for all $k \in \NN$. Note that $F \in \ker V$, hence $F \cdot V$ is $\CC$-complete and $\varphi_V^{F}$ is a holomorphic automorphism.
\item \[ G\left( \begin{pmatrix} a & b \\ c & d \end{pmatrix} \right) = g(b) \]
Note that $\ell \colon \NN \to \NN$ is injective and has in particular a closed discrete image. Again by the Mittag-Leffler interpolation theorem we find a holomorphic function $g \colon \CC \to \CC$ such that $g(\ell^{-1}(k)) = - \frac{1}{1 + f(k)}$ for all $k \in \NN$. This is well defined, since $f(k) \neq - 1$ for all $k \in \NN$ by construction. Note that $G \in \ker W$, hence $G \cdot W$ is $\CC$-complete and $\varphi_V^{F}$ is a holomorphic automorphism.
\end{enumerate}
We now check that we indeed interpolate correctly:
\begin{align*}
\begin{pmatrix} \lambda & k \\ 0 & \mu \end{pmatrix}
&\stackrel{\varphi_W^1}{\mapsto}
\begin{pmatrix} \lambda-k & k \\ \lambda - \mu - k & \mu + k\end{pmatrix} \\
&\stackrel{\varphi_V^F}{\mapsto}
\begin{pmatrix} (\lambda - \mu) f(k) + \lambda - k f(k) - k & -(1+f(k)) \cdot (-k -k f(k) + (\lambda - \mu) f(k)) \\ -k & -((\lambda - \mu) f(k) - \mu - k f(k) - k) \end{pmatrix} \\
&\stackrel{\varphi_W^G}{\mapsto} 
\begin{pmatrix} \lambda & -(1+f(k)) \cdot (-k - k f(k) + (\lambda - \mu) f(k)) \\ 0 & \mu \end{pmatrix} \\
&=
\begin{pmatrix} \lambda & \ell(k) \\ 0 & \mu \end{pmatrix} \qedhere
\end{align*}

\end{proof}

\begin{cor}
The spectral ball $\sball_2$ contains non-equivalent tame sets.
\end{cor}

\begin{remark}
The construction shows that for any discrete sequence in $\sdisc^2$ we can construct a sequence of tame sets in the respective fibres such that on every such tame set an injective selfmap can be interpolated by a holomorphic automorphism simultaneously. The construction is of course not limited to dimension two and works for any $n \geq 2$.
\end{remark}

\section{Relative Density Property for tame complements in $\CC^n$}
\label{secrelative}

\begin{remark}
The complex manifold $X := \CC^n \setminus A$, where $A$ is a tame discrete subset, does not have the density property: We can assume that $n \geq 2$. Since $A$ is discrete, every holomorphic function and every holomorphic vector field on $X$ extends holomorphically to $\CC^n$ due to the Hartogs' Extension Theorem. Again since $A$ is discrete, complete holomorphic vector fields must extend in $A$ as zero. Since linear combinations and Lie brackets of vector fields vanishing in $A$ will vanish in $A$ too, the Lie algebra generated by complete holomorphic vector fields in $X$ cannot be dense in the Lie algebra of all holomorphic vector fields on $X$.
\end{remark}

On the other hand, holomorphic automorphisms of $\CC^n \setminus A$ that are path-connected to the identity necessarily fix $A$, hence there is hope that they arise as (limit of) composition of flows of complete holomorphic vector fields of $\CC^n$ that vanish in $A$. This leads to the following definition that originally has been stated by Kutzschebauch, Leuenberger and Liendo \cite{MR3320241} and can be found also in \cite{FrancBook}*{Problem 4.11.1}.

\begin{definition}\cite{MR3320241}*{Definition 6.1}
Let $X$ be a complex manifold and let $A \subset X$ be a closed subvariety. We say that $X$ has the \emph{(strong) density property relative to $A$} if the Lie algebra generated by the complete holomorphic vector fields vanishing in $A$ is dense in the Lie algebra of all holomorphic vector fields vanishing in $A$.
\end{definition}

We notice that if $X$ is a Stein manifold, then Cartan's Theorem~B could be applied to the holomorphic functions resp.\ holomorphic vector fields vanishing in $A$ since they form a coherent sheaf. Therefore, this seems to be good notion to prove an Anders\'en--Lempert Theorem for this setup.

\begin{theorem}
\label{thm-reldens}
Let $A$ be a tame discrete subset of $\CC^n$. Then $\CC^n$ has the density property relative to $A$.
\end{theorem}

\begin{theorem}[\cite{MR3320241}*{Theorem 6.3}]
Let $X$ be a Stein manifold and $A \subset X$ be a closed subvariety such that $X$ has the density property relative to $A$.

Let $\Omega \subset X$ be an open subset and let $\varphi_t \colon \Omega \to X$ be a family of holomorphic injections with $\varphi_0 = \id_\Omega$, $\mathcal{C}^1$-smooth in $t \in [0, 1]$, such that every $\varphi_t(\Omega)$ is Runge in $X$ and $\varphi_t|(A \cap \Omega) = \id_{A \cap \Omega}$ for all $t \in [0, 1]$. Then for every compact $K \subset \Omega$ and every $\varepsilon > 0$ there exists a family of holomorphic automorphisms $\Phi_t \colon X \to X$, continuous in $t \in [0, 1]$ such that $\Phi_t|A = \id_A$, $\Phi_0 = \id_X$ and $\left| \varphi_1 - \Phi_1 \right|_K < \varepsilon$.
\end{theorem}

The statements of the following corollaries are well-known, but now arise naturally from the relative density property.

\begin{cor}
Let $A \subset \CC^n$ be a tame discrete subset. Then there exists a Fatou--Bieberbach domain $\Omega \subset \CC^n \setminus A$.
\end{cor}

\begin{cor}
Let $A \subset \CC^n$ be a tame discrete subset. Then $\CC^n \setminus A$ is holomorphically flexible in the sense of Arzhantshev et al.\ and hence an Oka manifold.
\end{cor}


Given two vector fields vanishing in $A$ to order $1$, their Lie bracket vanishes in $A$ to order at least $1$. However, in the Kaliman--Kutzschebauch formula which we will need in the proof of Theorem \ref{thm-reldens}, the difference of those two Lie brackets typically vanishes to order $2$ or higher. Hence, we need the next lemma which shows how complete vector fields vanishing in $A$ to order $1$ generate a Lie algebra that allows us to interpolate any term of order $1$ in the points of $A$.

\begin{lemma}
\label{lem-spanning}
Let $(z,w)$ be the coordinates of $\CC^2$ and let $a,b,c,d \colon \CC \to \CC$ be holomorphic functions. Then the following vector fields are complete:
\begin{align*}
 V  &:= a(z) \cdot \frac{\partial}{\partial w} \\
U  &:= b(z) \cdot w \frac{\partial}{\partial w} \\
W  &:= w \frac{\partial}{\partial z} \\
 V' &:= c((\sqrt{5} z - w)/\sqrt{5}) \cdot \left(\frac{\partial}{\partial z} + \sqrt{5} \frac{\partial}{\partial w} \right) \\
W' &:= d((\sqrt{3} z - w)/\sqrt{3}) \cdot \left(\frac{\partial}{\partial z} + \sqrt{3} \frac{\partial}{\partial w} \right)
\end{align*}
Taking $a \equiv 1$, we define
\[
Z := [b(z) \cdot V, W] + b'(z) \cdot U = \left[b(z) \cdot \frac{\partial}{\partial w}, w \frac{\partial}{\partial z}\right] + b'(z) \cdot w \frac{\partial}{\partial w}
 = b(z) \cdot \frac{\partial}{\partial z}
\]
For any holomorphic vector field $\Xi$ on $\CC^2$ that vanishes in $\ZZ \times \{ 0 \}$ we can choose these holomorphic functions $a,b,c,d \colon \CC \to \CC$ such that $V, Z, V',W'$ each vanish in $\ZZ \times \{ 0 \}$ to order $1$ and such that $\Xi - V -Z - V' - W'$ vanishes in $\ZZ \times \{ 0 \}$ to order $2$ or higher. 
\end{lemma}

\begin{proof}
The flow of a (constant) linear combination of partial derivatives is a translation, and hence exists for all complex times. If a complete vector field is multiplied by an element from its kernel or its second kernel, then the product is again a complete vector field.

Note that the vector field $Z$ is not complete, but contained in the Lie algebra generated by complete holomorphic vector fields.

We choose the holomorphic functions $a,b,c,d \colon \CC \to \CC$ to vanish precisely in $\ZZ$ to order $1$ and prescribe their derivatives in $j \in \ZZ$ to be equal to $a_j, b_j, c_j, d_j \in \CC$, respectively, as indicated below.

Without loss of generality, consider the point $(0,0)$, i.e.\ $j=0$, and the following power series expansions, assuming the derivatives were all prescribed to be equal to $1$:
\begin{align*}
V &= z \frac{\partial}{\partial w} + \dots \\
Z &= z \frac{\partial}{\partial z} + \dots \\
V' &= (z - w/\sqrt{5}) \cdot \left(\frac{\partial}{\partial z} + \sqrt{5} \frac{\partial}{\partial w}\right) + \dots \\
W' &= (z - w/\sqrt{3}) \cdot \left(\frac{\partial}{\partial z} + \sqrt{3} \frac{\partial}{\partial w}\right) + \dots
\end{align*}
Let the linear part of $\Xi$ in $(j,0)$ be given by $(\alpha z + \beta w) \frac{\partial}{\partial z} + (\gamma z + \delta w) \frac{\partial}{\partial w}$. Then this leads to the following system of linear equations:
\[
\begin{pmatrix}
0 & 1 & 1 & 1 \\
0 & 0 & -1/\sqrt{5} & -1/\sqrt{3} \\
1 & 0 & \sqrt{5} & \sqrt{3} \\
0 & 0 & -1 & -1
\end{pmatrix} \cdot \begin{pmatrix}
a_j \\ b_j \\ c_j \\ d_j
\end{pmatrix} = \begin{pmatrix}
\alpha(j) \\ \beta(j) \\ \gamma(j) \\ \delta(j)
\end{pmatrix}
\]
This matrix is invertible, hence we can find the desired constants $a_j,b_j,c_j,d_j \in \CC$.
\end{proof}

\begin{proof}[Proof of Theorem \ref{thm-reldens}]
For simplicity of notation, we show the statement only for $n=2$.
By $(z,w)$ we denote the coordinates of $\CC^2$. After a global holomorphic change of coordinates we may assume that $A = \ZZ \times \{0\}$. We use the vector fields and notation introduced in Lemma \ref{lem-spanning}. We may choose $b(z) = a(z) = \sin(2 \pi z)$ which vanishes precisely in $\ZZ$ to order $1$. The following vector fields are complete and vanish in $A$:
\begin{align*}
z^k \cdot V  \\
w^\ell \cdot W \\
z^k \cdot U = z^k w \cdot V
\end{align*}
We may now apply the Kaliman--Kutzschebauch formula:
\begin{align*}
[z^k \cdot U, w^\ell \cdot W] - [z^k \cdot V, w^{\ell+1} \cdot W]
 &= z^k \cdot w^\ell \cdot V(w) \cdot W \\
 &= z^k \cdot w^\ell \cdot \sin(2 \pi z) \cdot W
\end{align*}
Since this holds for all $k, \ell \in \NN_0$ we may take linear combinations and limits to obtain the non-trivial $\hol(\CC^2)$-module $\hol(\CC^2) \cdot \sin(2 \pi z) \cdot W$. Note its elements vanish in $A$ to order at least $2$.

The vector fields $V$ and $V'$ simultaneously vanish precisely in $A$, since $\sqrt{5}$ is irrational. The union of their zero sets is a variety of complex co-dimension one, in particular they span the tangent space almost everywhere. Outside of $A$ we can obviously flow along either $V$ or $V'$ to reach a point where the tangent space is spanned by $V$ and $V'$.

We may now apply Theorem 1 of Kaliman and Kutzschebauch \cite{KK-Criteria} and its proof also in the holomorphic setting to obtain every holomorphic vector field on $\CC^n$ that vanishes in $A$ to order $2$.

Given any holomorphic vector field $\Xi$ that vanishes in $A$ to order $1$, we first apply Lemma \ref{lem-spanning} to obtain a vector field $\Xi - \Theta$  that vanishes in $A$ to order $2$ or higher, where $\Theta$ is in the Lie algebra generated by complete vector fields vanishing in $A$.
\end{proof}

\section{Miscellanea}


In \cite{Varolin2000}*{Remark after Corollary~4.9} Varolin claims the following:
\begin{theorem}[False Theorem]
Let $X$ be a Stein manifold with the density property, $A=\{a_j\}_{j \in \NN}$, $B=\{b_j\}_{j \in \NN} \subset X$ closed discrete infinite sets. Then there exists a holomorphic injective map $F:X \to X$ such that $F(a_j)=b_j$ for all $j \in \NN$.
\end{theorem}

\begin{example}
Let $X=\CC^n$, $n>1$, $A$ a tame set and $B$ an unavoidable set. Suppose there exists $F$ as above, then $F(X \setminus A) \subset X \setminus B$. Since $X \setminus A$ is dominable there exists a non-degenerate holomorphic map $h \colon \CC^n \to X \setminus A$, hence $F \circ h$ is a non-degenerate holomorphic map avoiding $B$ but this is impossible as $B$ was unavoidable. 
\end{example}

\section*{Acknowledgments}

The authors would like to thank Frank Kutzschebauch for helpful remarks and suggestions. The work of the authors was supported in part by BI-DE/18-19-005 and a DAAD grant.

\end{document}